\documentclass[11pt]{article}
\usepackage{amsfonts,amssymb,amscd,amsmath,enumerate,verbatim,calc, times,url}
\usepackage{amsthm, psfrag,latexsym,epsfig,mdwlist,graphicx}
\usepackage[capitalise, noabbrev]{cleveref}
\usepackage{enumitem}
\usepackage{xcolor}
\usepackage{dirtytalk}
\usepackage{tikz}
\usetikzlibrary{positioning, shapes.geometric}
\usepackage{tkz-euclide}
\usepackage{stackengine}
\usepackage{nomencl}
\usepackage{txfonts}
\usetikzlibrary{shadings}
\theoremstyle{plain}
\newtheorem{theorem}{Theorem}[section]

\newtheorem{proposition}[theorem]{Proposition}
\newtheorem{corollary}[theorem]{Corollary}
\theoremstyle{definition}
\newtheorem{definition}[theorem]{Definition}

\newtheorem{example}[theorem]{Example}

\newtheorem{question}[theorem]{Question}

\newcommand{\qand}{\quad \mbox{ and } \quad}        
        
\newcommand{\qwhere}{\quad \mbox{ where } \quad}        
\newcommand{\T}{{\mathbb{T}}}                  
\newcommand{\M}{{\mathcal{M}}}                 
\newcommand{\G}{{\mathcal{G}}}                
\newcommand{\B}{{\mathcal{B}}}                
\newcommand{\rhk}[2]{\tilde{H}_{#1}(#2,k)}   
\newcommand{\xs}{x_1,\ldots,x_n}                
\newcommand{\m}{\mathbf{m}}   
\newcommand{\n}{\mathbf{n}}   
\newcommand{\q}{\mathbf{q}}   
\newcommand{\F}{{\mathcal{F}}}                  
\newcommand{\D}{\Delta}                         
\newcommand{\be}{\beta}                          
\newcommand{\lcm}{{\mathop{\rm{lcm}}}}          
\newcommand{\st}{\ : \ }                        
\newcommand{\tuple}[1]{\langle #1 \rangle}      
\newcommand{\LCM}{\mathop{\rm{LCM}}}
\newcommand{\dist}{\mathop{\rm{dist}}}
\newcommand{\facets}{\mathop{\rm{Facets}}} 
\newcommand{\NN}{\mathbb{N}} 
\renewcommand{\leq}{\leqslant}          
\renewcommand{\geq}{\geqslant}

\author{Sara Faridi\thanks{Department of
		Mathematics and Statistics, Dalhousie University, Halifax, Canada, 
		faridi@dal.ca}\thanks{Research supported by The Natural Sciences and Engineering Research Council of Canada (NSERC)}  \hspace{1in}
	Mayada Shahada \thanks{Mathematics Department, University of Bahrain, Kingdom of Bahrain, 
		MShahada@uob.edu.bh}}

\title{\Large \sc Well Ordered Covers, Simplicial
  Bouquets, and Subadditivity of Betti Numbers of Square-Free Monomial
  Ideals}

\begin{document}

\maketitle

\begin{abstract} Well ordered covers
  of square-free monomial ideals are subsets of the minimal generating
  set ordered in a certain way that give rise to a Lyubeznik
  resolution for the ideal, and have guaranteed
    nonvanishing Betti numbers in certain degrees. This paper is
  about square-free monomial ideals which have a well ordered
  cover. We consider the question of subadditivity of syzygies of
  square-free monomial ideals via complements in the lcm lattice of
  the ideal, and examine how lattice complementation breaks well
  ordered covers of the ideal into (well ordered) covers of
  subideals. We also introduce a family of well ordered
    covers called strongly disjoint sets of simplicial bouquets
    (generalizing work of Kimura on graphs), which are relatively easy to
    identify in simplicial complexes. We examine the subadditivity
    property via numerical characteristics of these bouquets.
  \end{abstract}

\section{Introduction}

This paper grew out of investigations into the subadditivity property
of syzygies of monomial ideals. For a homogeneous ideal $I$ of a
polynomial ring $S$, suppose the maximum degree $j$ such that
$\beta_{i,j}(S/I) \neq 0$ is denoted by $t_i$. The subadditivity
property is said to hold if we have $t_{a+b} \leq
  t_a+t_b$  for all positive values of $a$ and $b$.

While the subadditivity property or related inequalities are known to
hold in many special cases - certain cases for ideals of
  codimension $\leq 1$ (\cite[Corollary~4.1]{EHU}); some Koszul
rings~\cite{ACI}; when $I$ monomial ideal and $a=1$~\cite{HS}; certain
homological degrees for Gorenstein algebras~\cite{ElS}; when $a=1,2,3$
and $I$ monomial ideal generated in degree 2~\cite{FG,AN}; facet
ideals of simplicial forests~\cite{F1} - the problem is open for
monomial ideals and is known to fail (see Caviglia's
  example in~\cite[Example~4.5]{EHU}) for general homogeneous ideals.

In the case of monomial ideals, Betti numbers can be calculated as
ranks of homology modules of topological objects. In
  particular, the order complex of the lcm lattice of $I$ (the poset
  of least common multiples of the minimal monomial generating set of
  $I$ ordered by division) can be used for this purpose.  A
nonvanishing Betti number, in this context, corresponds
  by a result of Baclawski~\cite{B} to a ``complemented'' lcm lattice
(see \cref{s:subadditivity}).

This approach was initiated by the first author
in~\cite{F1}: the existence of complements which have nonvanishing
Betti numbers in the ``right'' homological degrees implies the
subadditivity property (\cref{q:main}). This approach was further
pursued by both authors in~\cite{FS}, where \cref{q:main} was
translated into the context of homological cycles in simplicial
complexes breaking into smaller ones.

This paper takes yet a different angle. The idea of complements really
comes down to the following: take a monomial $\m$ in the lcm lattice
of $I$. Among all the complements of $\m$ in the lcm lattice, can you
pick one, say $\m'$, which behaves desirably? If $\m$ is a square-free
monomial, this question is even simpler: $\m$ and $\m'$ correspond to
subsets $A$ and $A'$ of the variables $\{\xs\}$ such that $A\cup
A'=\{\xs\}$ and the product of the variables in $A \cap A'$ is not in
$I$. Can we consider the subideals induced on $A$ and $A'$ and extract
properties from/for them?

The main object of study using this approach will be ``well ordered
covers'' of ideals (\cref{d:woc}). The existence of a well ordered
cover of size $i$ is known, via the Lyubeznik resolution, to guarantee
a nonvanishing $i^{\text{th}}$ Betti number~\cite{EF2}. In this paper we
investigate when complements in the lcm lattice produce subideals with
well ordered covers of sizes $a$ and $b$ with $a+b=i$, in order to
result in the subadditivity property.

Moreover, we introduce \emph{strongly disjoint sets of
    simplicial bouquets}, which we show are always well ordered
  covers, and we demonstrate how they can be broken up to prove
  subadditivity in certain homological degrees. An advantage of
  simplicial bouquets is that they are rather easy to spot in
  simplicial complexes, as they do not rely on an ordering. Rather,
  one or more orderings are inherent in the definition (see
  \cref{t:bouquetsstrongly}).

In \cref{s:background} we set up the background on simplicial
resolutions. \cref{s:woc} introduces the reader to well ordered covers
of monomials. \cref{s:subadditivity} describes the subadditivity
property and contains one of the main results of the paper
(\cref{t:main}) which considers well ordered covers under
complementation. In \cref{s:bouquets} we introduce simplicial
  bouquets and show that certain types of simplicial bouquets are well
  ordered facet covers. We then apply the results of
  \cref{s:subadditivity} to simplicial complexes that contain strongly
  disjoint sets of bouquets, and show that the subadditivity property
  holds in degrees that come from the sizes of the bouquets
  (\cref{t:bouquets}). \cref{s:reordering} offers ways to optimize
the order of monomials in a well ordered cover to get the best
possible subadditivity results.

With some numerical manipulation, the results of this paper can be
adapted to non-square-free monomial ideals via polarization~\cite{Fr},
a method that transforms a monomial ideal into a square-free one which
retains many of the algebraic properties of the original ideal,
including the minimal free resolution.

\subsubsection*{Acknowledgements} The authors are grateful to the referees for
their comments, which helped improve this paper.
\section{Background}\label{s:background} 
\subsection{Simplicial complexes and facet ideals}

A {\bf simplicial complex} $\Delta$ on a finite vertex set $V(\Delta)$
is a set of subsets of $V(\Delta)$ such that $\{v\} \in \Delta$ for
every $v \in V(\Delta)$ and if $F \in \Delta$, then for every $G
\subseteq F$, we have $G \in \Delta$. The elements of $\Delta$ are
called {\bf faces}; the maximal faces with respect to inclusion are
called {\bf facets}, and a simplicial complex contained in $\Delta$ is
called a {\bf subcomplex} of $\Delta$. The set of all facets in
$\Delta$ defines $\Delta$ and is denoted by
$\facets(\Delta)$. If $\facets(\Delta)=\{F_1,\ldots,F_q\}$, then we
write $\Delta=\tuple{F_1,\ldots,F_q}$. 

A {\bf subcollection} of $\Delta$ is a subcomplex of
  $\Delta$ whose facets are also facets of $\Delta$. If $A \subseteq
  V(\Delta)$, then the {\bf induced subcollection} $\Delta_{[A]}$ is
  the simplicial complex defined as $$\Delta_{[A]}= \left \langle F
  \in \facets(\Delta) \st F \subseteq A \right \rangle.$$

We say a facet $F$ contains a {\bf free vertex} $v$ if $v \notin G$
for every facet $G \in \facets(\Delta) \setminus \{F\}$.

Given a simplicial complex $\Delta$ on vertices $\{x_1,\ldots,x_n\}$, we can define the {\bf facet ideal} of $\Delta$ as
	$$\F(\Delta)= \big (\prod_{x_i \in F}x_i \st  F \mbox{ is a
  facet of } \Delta \big )$$ which is an ideal of
$S=k[x_1,\ldots,x_n]$. Conversely, given a square-free monomial ideal
$I \subset S$, the {\bf facet complex} of $I$ is the
simplicial complex
$$\F(I)=\big \langle  F \st \prod_{x_i \in F} x_i  \text{ is a generator of } I \big \rangle.$$

\begin{example}\label{e:running-1}
  
	For $I=(xy,yz,zu)$, the simplicial complex $\F(I)$ is below.
	
	
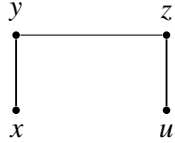
\begin{center}
		\begin{tikzpicture}	
		\coordinate (x) at (0,0);
		\coordinate (y) at (0,1);
		\coordinate (z) at (2,1);
		\coordinate (u) at (2,0);
		
		\draw (x)node[circle,fill,inner sep=1pt,label=below:$x$](x){};
		\draw (y)node[circle,fill,inner sep=1pt,label=above:$y$](y){};
		\draw (z)node[circle,fill,inner sep=1pt,label=above:$z$](z){};
		\draw (u)node[circle,fill,inner sep=1pt,label=below:$u$](u){};
		\draw [-,line width=0.7pt](x)--(y);
		\draw (y)--(z);
		\draw [-,line width=0.7pt](z)--(u);
		\end{tikzpicture}
	\end{center}

\end{example}

\subsection{Simplicial Resolutions}

Any monomial ideal $I$ of $S=k[\xs]$ admits a minimal
  {\bf graded free resolution}, which is an exact sequence of free
  $S$-modules
$$ 0 \to \oplus_{j \in \NN} S(-j)^{\beta_{p,j}}\to \oplus_{j \in \NN}
S(-j)^{\beta_{p-1,j}} \to \cdots \to \oplus_{j \in \NN}
S(-j)^{\beta_{1,j}} \to S.$$ For each $i$ and $j$, the rank
$\beta_{i,j}(S/I)$ of the free $S$-modules appearing above are called
the {\bf graded Betti numbers} of the $S$-module $S/I$, and the {\bf
  total Betti number} in homological degree $i$ is $$\be_{i}(S/I)=\sum_j
\be_{i,j}(S/I).$$

If $I$ is generated by monomials, the graded Betti
  numbers can be further refined into sums of \emph{multigraded Betti
    numbers}.  For a monomial $\m$ in $S$, the {\bf multigraded Betti
  number} of $S/I$ is of the form $\be_{i,\m}(S/I)$ and we have
\begin{equation}\label{e:multigraded}
  \be_{i,j}(S/I)=\sum \be_{i,\m}(S/I)
\end{equation}
where the sum is taken over all monomials $\m$ of degree
  $j$ that are least common multiples of subsets of the minimal
  monomial generating set of $I$.

The multigraded Betti numbers $\be_{i,\m}(S/I)$ are related to
the combinatorics of the ideal $I$. Given a monomial
ideal $I$ minimally generated by $\m_1,\ldots,\m_q$,
one can consider a simplicial complex $\Gamma$ on $q$ vertices
$\{v_1,\ldots,v_q \}$, where each vertex $v_i$ is labeled with the
monomial generator $\m_i$, and each face $\tau$ of $\Gamma$ is
labeled with the monomial
$$\lcm(\tau)=\lcm \big (\m_i\st\m_i \in \tau \big ).$$

We say that $\Gamma$ {\bf supports a free resolution of $I$} if the
simplicial chain complex of $\Gamma$ can be ``homogenized'', using the
monomial labels of the faces, to produce a free resolution of $I$. For
details of homogenization of a chain complex see~\cite[Section~55]{P}. The
resulting free resolution is called a {\bf simplicial resolution}.

{\bf Taylor's resolution} (\cite{T}, see also~\cite[Construction~26.5]{P}) is an example of a simplicial
resolution where the underlying simplicial complex is a full simplex
$\T(I)$ over the vertex set labeled with the monomial generators
$\{\m_1,\ldots,\m_q\}$ of $I$, called the \textbf{Taylor complex} of
$I$. It is known that all simplicial complexes supporting a free
resolution of $I$ are subcomplexes of the Taylor complex~(\cite{M}),
in other words: all simplicial resolutions are contained in the Taylor
resolution.

If $\Gamma$ is a simplicial complex supporting a free
resolution of $I$, and $\m$ is a monomial in $S$, the simplicial
subcomplex $\Gamma_{<\m}$ is defined as
	\begin{center}
	  $\Gamma_{<\m}=\{ \tau \in \Gamma\st \lcm(\tau)
          \mbox{ strictly divides } \m\}$.
	\end{center}

\begin{example}\label{e:running-2} Consider $I=(xy,yz,zu)$ in $k[x,y,z,u]$.
     The Taylor complex $\T(I)$ and a subcomplex
          $\T(I)_{<xyzu}$ are shown in the following figures.
	
	\vspace{0.5cm}
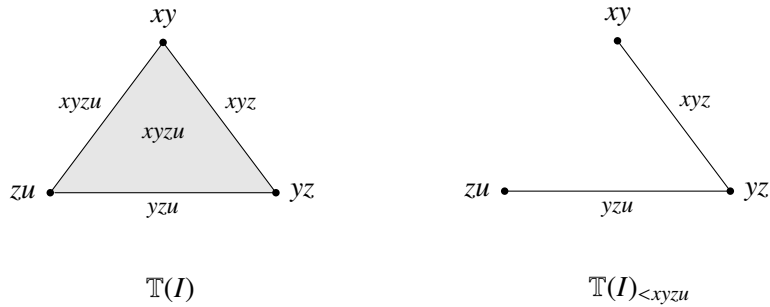
	\begin{minipage}{.2\textwidth}
		\begin{tikzpicture}	
		\coordinate (ac) at (-1.5,0);
		\coordinate (ae) at (0,2);
		\coordinate (ab) at (1.5,0);
		\coordinate (C) at (0,1);

		\draw [top color=gray!20,bottom color=gray!20](ac)node[circle,fill,inner sep=1pt,label=left:$zu$](ac){}--(ae)node[circle,fill,inner sep=1pt,label=above:$xy$](ae){}--(ab)node[circle,fill,inner sep=1pt,label=right:$yz$](ab){}--(ac);

		\draw (C)node[anchor=midway,below]{\footnotesize{$xyzu$}};
		
		\tkzLabelSegment[above left=-2pt and -2pt](ac,ae){\footnotesize{$xyzu$}}  
		\tkzLabelSegment[above right=-2pt and -2pt](ae,ab){\footnotesize{$xyz$}}  
		\tkzLabelSegment[below=0pt and -2pt](ac,ab){\footnotesize{$yzu$}}
		
		\node [below=1cm, align=flush center,text width=8cm] at (0.1,0){$\T(I)$};
		\end{tikzpicture}
		
	\end{minipage}
	\hspace{3cm}
	\begin{minipage}{.2\textwidth}
		\begin{tikzpicture}

		\coordinate (ac) at (-1.5,0);
		\coordinate (ae) at (0,2);

		\coordinate (ab) at (1.5,0);
		\coordinate (C) at (0,0.85);

		\draw (ac)node[circle,fill,inner sep=1pt,label=left:$zu$](ac){}--(ab)node[circle,fill,inner sep=1pt,label=right:$yz$](ab){}--(ae)node[circle,fill,inner sep=1pt,label=above:$xy$](ae){};

		\tkzLabelSegment[above right=-2pt and -2pt](ae,ab){\footnotesize{$xyz$}}  
		\tkzLabelSegment[below=0pt and -2pt](ac,ab){\footnotesize{$yzu$}}
		
		\node [below=1cm, align=flush center,text width=8cm] at (0.3,0){$\T(I)_{<xyzu}$};
		
		\end{tikzpicture}
	\end{minipage}
\end{example}

If $\Gamma$ supports a free resolution of a
monomial ideal $I$, then for a fixed integer $i$, the Betti number
$\be_{i}(S/I)$ is bounded above by the number of $(i-1)$-faces of
$\Gamma$. Therefore, the more we \say{shrink} the supporting complex
$\Gamma$, the better we bound the Betti numbers.

In particular, as stated by Bayer and
  Sturmfels~\cite{BS}, the multigraded Betti numbers of $I$ can be
determined by the dimensions of reduced homologies of subcomplexes of
$\D$. The statement below is from Peeva's
  textbook~\cite{P}.

\begin{theorem}[\cite{P},~Theorem~57.6]\label{t:BS}
  Let $I$ be a proper monomial ideal of $S$ which is minimally
  generated by the monomials $\m_1,\ldots,\m_q$, and
    suppose that $I$ has a free resolution supported on a simplicial
    complex $\Gamma$. For $i> 0$ and a monomial $\m$ of positive
    degree, the multigraded Betti numbers of $I$ are given by
	\begin{align}
	\nonumber \be_{i,\m}(S/I)&=\left\{ \begin{array}{ccl}
          \dim_k\rhk{i-2}{\Gamma_{<\m}} & \mbox{if} &
          \m \mid \lcm(\m_1,\ldots,\m_q)\\0 &
          \mbox{ } & \text{otherwise.} \end{array} \right.
	\end{align}
\end{theorem}

\begin{example} 
	For $I=(xy,yz,zu)$ in \cref{e:running-2},
        $\T(I)_{<xyzu}$ is acyclic and hence
        $\be_{i,xyzu}(S/I)=0$ for all $i$.
\end{example}

For a monomial ideal $I$ of $S$ which is minimally generated by the
set $\G=\{\m_1,\ldots,\m_q\}$, the {\bf lcm
  lattice} of $I$, denoted by $\LCM(I)$, is the bounded lattice whose
elements are the least common of subsets of $\G$ ordered by
divisibility.  The {\bf top} element of $\LCM(I)$ is
$\hat{1}=\lcm(\m_1,\ldots,\m_q)$ and the {\bf bottom}
element is $\hat{0}=1$ regarded as the $\lcm$ of the empty set. The
least common multiple of elements in $\LCM(I)$ is their {\bf join}.

\begin{example}
	The following is the $\LCM(I)$ for $I=(xy,yz,zu)$ from \cref{e:running-2}.
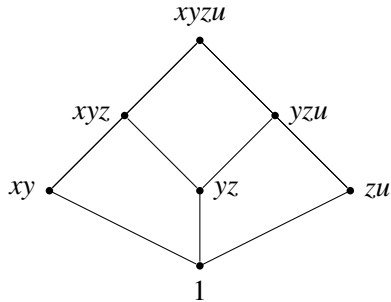
	\begin{center}
	\begin{tikzpicture}	
		\coordinate (id) at (0,0);
		\coordinate (xy) at (-2,1);
		\coordinate (yz) at (0,1);
		\coordinate (zu) at (2,1);
		\coordinate (xyz) at (-1,2);
		\coordinate (yzu) at (1,2);	
		\coordinate (xyzu) at (0,3);

		\draw (id)node[circle,fill,inner sep=1pt,label=below:$1$](id){}--(zu)node[circle,fill,inner sep=1pt,label=right:$zu$](zu){};
		\draw (id)--(yz)node[circle,fill,inner sep=1pt,label=right:$yz$](yz){};
		\draw (id)--(xy)node[circle,fill,inner sep=1pt,label=left:$xy$](xy){};		
		\draw (zu)--(yzu)node[circle,fill,inner sep=1pt,label=right:$yzu$](yzu){};	
		\draw (yz)--(yzu){};
		\draw (xy)--(xyz)node[circle,fill,inner sep=1pt,label=left:$xyz$](xyz){};
		\draw (yz)--(xyz){};
		\draw (yzu)--(xyzu)node[circle,fill,inner sep=1pt,label=above:$xyzu$](xyzu){};
		\draw (xyz)--(xyzu){};			
		\draw (zu)--(xyzu){};			
		\draw (xy)--(xyzu){};			
	\end{tikzpicture}
	\end{center}
	
	\end{example}

\section{(Well ordered) covers}\label{s:woc}

Let $\Delta$ be a simplicial complex.  A set $D \subseteq
\facets(\Delta)$ is called a {\bf facet cover} of $\Delta$ if every
vertex $v$ of $\Delta$ belongs to some facet $F$ in $D$. A facet cover
is called {\bf minimal} if no proper subset of it is a facet cover of
$\Delta$. For instance, the simplicial complex in \cref{e:running-1}
has the set $\{\{x,y\},\{z,u\}\}$ as a minimal facet cover.

	
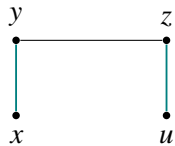
\begin{center}
		\begin{tikzpicture}	
		\coordinate (x) at (0,0);
		\coordinate (y) at (0,1);
		\coordinate (z) at (2,1);
		\coordinate (u) at (2,0);
		
		\draw (x)node[circle,fill,inner sep=1pt,label=below:$x$](x){};
		\draw (y)node[circle,fill,inner sep=1pt,label=above:$y$](y){};
		\draw (z)node[circle,fill,inner sep=1pt,label=above:$z$](z){};
		\draw (u)node[circle,fill,inner sep=1pt,label=below:$u$](u){};
		\draw [-,line width=0.7pt,color=teal](x)--(y);
		\draw (y)--(z);
		\draw [-,line width=0.7pt,color=teal](z)--(u);
		\end{tikzpicture}
	\end{center}

We can translate a minimal facet cover of a simplicial
  complex to its facet ideal.  If $I=(\m_1,\ldots, \m_q)$ is a
square-free monomial ideal and $\Delta=\F(I)=\langle F_1,\ldots,F_q
\rangle$ so that each $\m_i$ is the product of the vertices in $F_i$,
then we say $\m_{i_1},\ldots,\m_{i_t}$ is a {\bf (minimal) cover of
  $I$} to imply that $F_{i_1},\ldots,F_{i_t}$ is a (minimal) facet
cover of $\Delta$.

For the sake of simplicity, assume that every variable $x_i$ appears
in at least one of the generators of $I$. If $\M=\{\m_{i_1},\ldots
,\m_{i_t}\}$ is a minimal cover of $I$, then one can see that for
every $j \in \{1,\ldots , t \}$
$$\lcm(\m_{i_1},\ldots ,\widehat{\m}_{i_j},\ldots ,\m_{i_t})\neq
\lcm(\m_{i_1},\ldots ,\m_{i_t})=x_1 \ldots x_n.$$ This implies that if
$\T(I)$ is the Taylor complex of $I$, then $\T(I)_{<x_1 \ldots x_n}$
contains the boundary of a hollow $(t-2)$-cycle, which by \cref{t:BS}
means that we could potentially have $\beta_{t,x_1\ldots x_n}(S/I)
\neq 0$.

In other words, the existence of a minimal cover $\M$ of length $t$
indicates that we might have a ``top degree'' Betti number
$\beta_{t,n}$. By ordering the generators of $I$ we can make $\M$ into
a ``well ordered'' cover, which, using a Lyubeznik resolution (a
simplicial resolution that is based on ordering the generators of an
ideal~\cite{L}), \emph{guarantees} the nonvanishing of $\beta_{t,n}$
(\cref{t:Nursel}).

\begin{definition}[\cite{EF2},~Definition 3.1]\label{d:wofc}
  Let $\Delta$ be a simplicial complex. A sequence of
    $F_1,\ldots,F_s$ of facets of $\Delta$ is called a {\bf well
      ordered facet cover} if it is minimal facet cover of $\Delta$,
    and for every facet $F^\prime \notin \{F_1,\ldots,F_s\}$ of
    $\Delta$ there exists $j \leq s-1$ such that $$F_{j} \subseteq
    F^\prime \cup F_{j+1} \cup \cdots \cup F_s.$$
\end{definition}

  The definition below is an equivalent version of \cref{d:wofc}
  stated for monomial ideals.

\begin{definition}[{\bf Well ordered cover}]\label{d:woc}
        A sequence $\m_1,\ldots,\m_s$ of generators of a
          square-free monomial ideal $I$ is called a {\bf well
          ordered cover} of $I$ if $\{\m_1,\ldots,\m_s\}$ is a minimal
        cover of $I$ and for every generator $\m^\prime \notin
        \{\m_1,\ldots,\m_s\}$ of $I$ there exists $j \leq s-1$ such
        that $$\m_{j} \mid \lcm(\m^\prime,\m_{j+1},\ldots,\m_s).$$
\end{definition}

\begin{example}

	\begin{enumerate}
		
		\item For $I=(abz,bcz,xyz,axz)$, $\{abz,bcz,xyz\}$ is a well ordered cover since $abz$ divides $\lcm(axz,bcz,xyz)$.
		
		
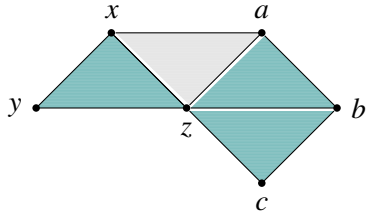
\begin{center}
			\begin{tikzpicture}	
			\coordinate (x) at (-1,1);
			\coordinate (y) at (-2,0);
			\coordinate (z) at (0,0);
			\coordinate (a) at (1,1);
			\coordinate (b) at (2,0);
			\coordinate (c) at (1,-1);
			
			\draw [top color=teal!50,bottom color=teal!50](x)node[circle,fill,inner sep=1pt,label=above:$x$](x){}--(y)node[circle,fill,inner sep=1pt,label=left:$y$](y){}--(z)node[circle,fill,inner sep=1pt,label=below:$z$](z){}--cycle;
			
			\draw [top color=gray!20,bottom color=gray!20](z)--(a)node[circle,fill,inner sep=1pt,label=above:$a$](a){}--(x)--(z);

			\draw [top color=teal!50,bottom color=teal!50](z)--(b)node[circle,fill,inner sep=1pt,label=right:$b$](b){}--(a)--cycle;
			
			\draw [top color=teal!50,bottom color=teal!50](z)--(c)node[circle,fill,inner sep=1pt,label=below:$c$](c){}--(b)--cycle;
			
			\end{tikzpicture}
		\end{center}
		\item For $I=(xy,yz,zu)$, $\{xy,zu\}$ is a minimal cover that is not a well ordered cover of $I$ since $xy \nmid \lcm(yz,zu)$ and $zu \nmid \lcm(yz,xy)$.
		
		
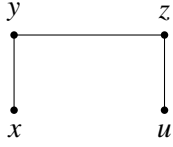
\begin{center}
			\begin{tikzpicture}	
			\coordinate (x) at (0,0);
			\coordinate (y) at (0,1);
			\coordinate (z) at (2,1);
			\coordinate (u) at (2,0);

			\draw (x)node[circle,fill,inner sep=1pt,label=below:$x$](x){}--(y)node[circle,fill,inner sep=1pt,label=above:$y$](y){}--(z)node[circle,fill,inner sep=1pt,label=above:$z$](z){}--(u)node[circle,fill,inner sep=1pt,label=below:$u$](u){};
			\end{tikzpicture}
		\end{center}
Notice that $I$ has no well ordered cover since $\{xy,zu\}$ is the
only possible minimal cover and it is not well ordered.
		
	\end{enumerate}
\end{example}

A class of examples of well ordered facet covers is
  \emph{(simplicial) bouquets}, which will be discussed in
  \cref{s:bouquets}.

  As the following theorem shows, well ordered covers are facets
in a Lyubeznik complex, and hence ensure
nonvanishing multigraded Betti numbers.

\begin{theorem}[\cite{EF2}]\label{t:Nursel} 
	Let $\mathcal{M}=\{\m_{1},\ldots,\m_s\}$ be a well ordered
        cover of a square-free monomial ideal $I$. Then there is a total
        order $<$ on the set of generators of $I$ such that
        $\mathcal{M}$ is a facet of the Lyubeznik simplicial complex
        and hence $\be_{s,\m}(S/I)\neq 0$ where
        $\m=\lcm(\m_1,\ldots,\m_s)$.
\end{theorem}

The converse of the above theorem holds in some cases, for example
when $I$ is the facet of a simplicial forest $\Delta$, every nonzero Betti number of $I$ corresponds to a well ordered
  facet cover of an induced subcollection of $\Delta$~(\cite{EF2}).

\section{The subadditivity property}\label{s:subadditivity}

In this section, we will explore how we can use well ordered 
covers to consider the {\bf subadditivity property} for the maximal
degrees of syzygies of square-free monomial ideals. Let $$t_a=\max \{j
\st \be_{a,j}(S/I) \neq 0\}.$$ We say that $I$ satisfies the
subadditivity property if for all $a,b>0$ with $a+b \leq$ the
projective dimension of $S/I$, $$t_{a+b} \leq t_a + t_b.$$

In what follows, we will be moving back and forth between a
square-free monomial ideal $I$ and its facet complex $\D=\F(I)$. For a
monomial $\m \in \LCM(I)$ where $\m=x_{i_1}\cdots x_{i_h}$, by
$I_{[\m]}$ we mean the facet ideal of the induced subcollection
$\D_{[\m]}= \D_{[\{x_{i_1},\ldots, x_{i_h}\}]}$  or in other words
$$I_{[\m]}=\F(\D_{[\m]})=\F(\D_{[\{x_{i_1},\ldots, x_{i_h}\}]}).$$ 

If we set $\Gamma$ and $\Gamma'$ to be the Taylor
  complexes of $I$ and $I_{[\m]}$, respectively, then \cref{t:BS}
    indicates that
    \begin{equation}\label{e:induced}
      \beta_{i,\m}(S/I)=\beta_{i,\m}(S/I_{[\m]})
    \end{equation} as
    $\Gamma_{<\m}= \Gamma'_{<\m}$. For an integer $i$,
    \eqref{e:multigraded} and \eqref{e:induced} show that
$$\beta_{i}(S/I) \neq 0 \iff \beta_{i,\m}(S/I) \neq 0 \mbox{ for some
} \m \in \LCM(I) \iff \beta_{i}(S/I_{[\m]}) \neq 0 \mbox{ for some }
\m \in \LCM(I).$$ In the rightmost statement in the previous line, the
monomial $\m$ is at the ``top'' of the lcm lattice of $I_{[\m]}$.
  This useful observation tells us that questions about
    multidegree Betti numbers can be reduced to questions about ``top
    degree'' Betti numbers. In the same vein, the subadditivity
  question can always be rephrased as a ``top degree'' one. We state
  this version for square-free monomials below.

\begin{question}[{\bf Top degree subadditivity}]\label{q:topdegree}
  Suppose $I=(\m_1,\ldots,\m_q)$ is a square-free monomial ideal in a
  polynomial ring $S$, and $\lcm(\m_1,\ldots,\m_q)=x_{i_1}\cdots
  x_{i_r}$. Suppose $\beta_{i,r} (S/I)\neq 0$, and $a,b >0$ are such
  that $i=a+b$. Then can we show that $t_a+t_b \geq r =t_{a+b}$?
\end{question}

In fact, the ambient ring does not really matter, so for the sake of
simplicity we can assume in \cref{q:topdegree} that $r=n$ and
$x_{i_1}\cdots x_{i_r}=x_1\cdots x_n$. If $I$ is a square-free
monomial ideal in $S=k[\xs]$, then since there is only one monomial of degree $n$ in $\LCM(I)$, from \eqref{e:multigraded} we have $$\beta_{i,n}(S/I)=\beta_{i,x_1\cdots x_n}(S/I).$$ So from now on we will  start from this setting.

For a square-free monomial ideal $I$ in the variables
$x_1,\ldots,x_n$, two monomials $\m,\m' \in \LCM(I)$ are called {\bf
  lattice complements} if $\lcm(\m,\m')=x_1\cdots x_n$ and
$\gcd(\m,\m')  \notin I$. As a potential way to examine the subadditivity
property, the first author raised the following question in \cite{F1}
(see also~\cite{FS}):

\begin{question}[{\bf Betti numbers of lattice complements}~\cite{F1},~Question~1.1]\label{q:main} If $I$ is a square-free monomial ideal
	in variables $\xs$, and $\beta_{i,n} (S/I)\neq 0$, $a,b >0$ and
	$i=a+b$, are there complements $\m$ and $\m'$ in $\LCM(I)$ with
	$\beta_{a,\m}(S/I) \neq 0$ and $\beta_{b,\m'}(S/I) \neq 0$?
\end{question}
The existence of lattice complements will establish the subadditivity property of $I$ simply since
\begin{center}
	$t_a+t_b \geq \deg(\m)+ \deg(\m') \geq n=t_{a+b}$.
\end{center}

A well ordered cover seems to be a promising place to look for lattice
complements as the following example shows.

\begin{example}
	Let $\Delta=\F(I)=\langle xy,yz,xz,za,ab,bc \rangle$.
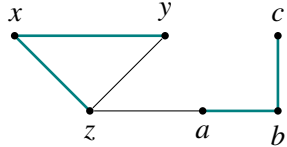
	\begin{center}
		\begin{tikzpicture}	
			\coordinate (x) at (-1,1);
			\coordinate (y) at (1,1);
			\coordinate (z) at (0,0);
			\coordinate (a) at (1.5,0);
			\coordinate (b) at (2.5,0);
			\coordinate (c) at (2.5,1);

			\draw [-,line width=1pt,color=teal](x)--(y);
			
			\draw [-,line width=1pt,color=teal](x)--(z);
			
			\draw (x)node[circle,fill,inner sep=1pt,label=above:$x$](x){};
			
			\draw (y)node[circle,fill,inner sep=1pt,label=above:$y$](y){};
			
			\draw (z)node[circle,fill,inner sep=1pt,label=below:$z$](z){};
			
			\draw (y)--(z);
			
			\draw (z)--(a);
			
			\draw [-,line width=1pt,color=teal](a)--(b);
			
			\draw [-,line width=1pt,color=teal](b)--(c);
			
			\draw (a)node[circle,fill,inner sep=1pt,label=below:$a$](a){};
			
			\draw (b)node[circle,fill,inner sep=1pt,label=below:$b$](b){};

			\draw (c)node[circle,fill,inner sep=1pt,label=above:$c$](c){};
			
		\end{tikzpicture}
	\end{center}

According to Macaulay2~\cite{M2}, $S/I$ has the following Betti table:
	$$\begin{array}{rllllll}
		&0 &1 &2 &3 &4\\
		\mbox{total}:&1 &6 &10 &7 &2\\
		0:&1 &. &. &. &.\\
		1:&. &6 &6 &1 &. \\
		2:&. &. &4 &6 &2
	\end{array}$$

	Here $\be_{4,xyzabc}(S/I)\neq 0$ and hence $t_4=6$. It is easy to check that $\mathcal{M}=\{ab,xy,bc,xz\}$ (highlighted above) is a well ordered cover. We consider the following two cases:
	
	\begin{enumerate}
		\item $a=1$ and $b=3$. Then, using the above Betti table $t_1=2$ and $t_3=5$. Here we have $t_4<t_1+t_3=7$. Define the monomials $\m,\m' \in \LCM(I)$ as $\lcm$ of some subsets of $\mathcal{M}$ as follows: 
		\begin{center}
			$\m=ab=\lcm(ab)$ and $\m'=bcxyz=\lcm(xy,bc,xz)$. 
		\end{center}
		
		\begin{minipage}{.2\textwidth}
			\begin{center}
				\begin{tikzpicture}	
					\coordinate (a) at (0,0);
					\coordinate (b) at (1,0);

					
					
					
					
					
					
					
					\draw (a)--(b);
					
					
					\draw (a)node[circle,fill,inner sep=1pt,label=below:$a$](a){};
					
					\draw (b)node[circle,fill,inner sep=1pt,label=below:$b$](b){};


					\node [below=1cm, align=flush center,text width=8cm] at (0.7,-1.3){$\D_{[\m]}$};
				\end{tikzpicture}
			\end{center}
		\end{minipage}
		\hspace{3cm}
		\begin{minipage}{.2\textwidth}
			\begin{center}
				\begin{tikzpicture}	
					\coordinate (x) at (-1,1);
					\coordinate (y) at (1,1);
					\coordinate (z) at (0,0);
					\coordinate (b) at (2.5,0);
					\coordinate (c) at (2.5,1);	
					
					\draw (x)--(y);
					
					\draw (x)--(z);
					
					\draw (x)node[circle,fill,inner sep=1pt,label=above:$x$](x){};
					
					\draw (y)node[circle,fill,inner sep=1pt,label=above:$y$](y){};
					
					\draw (z)node[circle,fill,inner sep=1pt,label=below:$z$](z){};
					
					\draw (y)--(z);
					
					
					
					\draw (b)--(c);
					
					
					\draw (b)node[circle,fill,inner sep=1pt,label=below:$b$](b){};

					\draw (c)node[circle,fill,inner sep=1pt,label=above:$c$](c){};

					\node [below=1cm, align=flush center,text width=8cm] at (0.5,0){$\D_{[\m^\prime]}$};
					
				\end{tikzpicture}
			\end{center}
		\end{minipage}

		Then, 
		$$\be_{1,ab}(S/I)=\dim_k \rhk{-1}{\emptyset}\neq
		0 \qand \be_{3,xyzbc}(S/I)=\dim_k
		\rhk{1}{\langle xy,yz,xz,bc\rangle}\neq 0.$$
		
		Note that $\lcm(\m,\m')=xyzabc$ and $\gcd(\m,\m')=b \notin
		I$. As a result, $\m$ and $\m'$ are lattice complements and
		$$t_1+t_3 = \deg(\m)+\deg(\m') >6 =t_4.$$ 
		
		\item $a=b=2$. From the Betti table  we have $t_2=4$ and $t_4<2t_2=8$.
		As in the above case, we take
		$$\m=abxy=\lcm(ab,xy) \qand  \m'=bcxz=\lcm(bc,xz).$$
		
		\begin{minipage}{.1\textwidth}
			
			\begin{tikzpicture}	
				\coordinate (x) at (-1,1);
				\coordinate (y) at (1,1);
				\coordinate (a) at (2,0);
				\coordinate (b) at (3,0);
				
				\draw (x)--(y);
				
				
				\draw (x)node[circle,fill,inner sep=1pt,label=above:$x$](x){};
				
				\draw (y)node[circle,fill,inner sep=1pt,label=above:$y$](y){};
				
				
				
				
				\draw (a)--(b);
				
				
				\draw (a)node[circle,fill,inner sep=1pt,label=below:$a$](a){};
				
				\draw (b)node[circle,fill,inner sep=1pt,label=below:$b$](b){};


				\node [below=1cm, align=flush center,text width=8cm] at (0.7,0){$\D_{[\m]}$};
			\end{tikzpicture}
			
		\end{minipage}
		\hspace{5cm}
		\begin{minipage}{.1\textwidth}
			
			\begin{tikzpicture}	
				\coordinate (x) at (-1,1);
				\coordinate (z) at (0,0);
				\coordinate (b) at (2.5,0);
				\coordinate (c) at (2.5,1);	
				
				
				\draw (x)--(z);
				
				\draw (x)node[circle,fill,inner sep=1pt,label=above:$x$](x){};
				
				
				\draw (z)node[circle,fill,inner sep=1pt,label=below:$z$](z){};
				
				
				
				
				\draw (b)--(c);
				
				
				\draw (b)node[circle,fill,inner sep=1pt,label=below:$b$](b){};

				\draw (c)node[circle,fill,inner sep=1pt,label=above:$c$](c){};

				\node [below=1cm, align=flush center,text width=8cm] at (1,0){$\D_{[\m^\prime]}$};
				
			\end{tikzpicture}
			
		\end{minipage}

		Here, we also have $\lcm(\m,\m')=xyzabc$ and $\gcd(\m,\m')=bx \notin I$. So $\m$ and $\m'$ are lattice complements, and we have 
		$$\be_{2,xyab}(S/I)=\dim_k \rhk{0}{\langle xy,ab \rangle}\neq 0 \qand
		\be_{2,xzbc}(S/I)=\dim_k \rhk{0}{\langle xz,bc\rangle}\neq 0.$$
		
		As a result,
		$$2t_2 = \deg(\m)+\deg(\m') > 6=t_4.$$
	\end{enumerate}
	
\end{example}

\begin{proposition}\label{p:complements}
	Let $\mathcal{M}=\{\m_1,\ldots,\m_s\}$ be a well ordered cover of a square-free monomial ideal $I$. Define for each $1 \leq a \leq s-1$ the monomials 
	\begin{center}
	  $\m=\lcm(\m_1,\ldots,\m_a)$ \hspace{3mm} and
          \hspace{3mm} $\m'=\lcm(\m_{a+1},\ldots,\m_s)$.
		\end{center}
	Then $\m$ and $\m'$ are lattice complements in $\LCM(I)$.
	\end{proposition}

\begin{proof} 	Suppose $I=(\m_1,\ldots,\m_s,\n_1,\ldots,
  \n_k)$ and $$\lcm(\m_1,\ldots,\m_s)=\lcm(\m,\m')=x_1 \cdots
  x_n.$$ Suppose that $\m$ and $\m'$ are not lattice complements, so
  $\gcd(\m,\m') \in I$. In particular, there is a generator
  $\q$ of $I$ such that $$\q \mid \lcm(\m_1,\ldots ,
  \m_a) \qand \q \mid \lcm (\m_{a+1},\ldots, \m_s).$$

  Suppose that $\q=\m_i$ for some $i=1,\ldots,s$. If $i \leq a$, then as
  $\m_i \mid \lcm (\m_{a+1},
  \ldots, \m_s)$ we must have 
	$$\lcm(\m_1, \ldots, \widehat{\m}_i, \ldots, \m_a, \m_{a+1},
  \ldots, \m_s)=\lcm(\m_1, \ldots, \m_s),$$ which contradicts
  $\mathcal{M}$ being a minimal cover of $I$.

  The case $i \geq a+1$ also leads to the same contradiction.  Suppose
  that $\q=\n_i$ for some $i\in \{1,\ldots,k\}$. Since
  $\mathcal{M}$ is a well ordered cover of $\Delta$, then there
  exists $\ell \leq s-1$ such that
	$$\m_\ell \mid \lcm (\n_i , \m_{\ell+1} , \ldots
  , \m_s.)$$

If $\ell \leq a$, then $\ell +1 \leq a+1$ and as $\n_i \mid \lcm(\m_{a+1}, \ldots, \m_s)$ we have
	$$\m_\ell \mid \lcm(\m_{\ell+1}, \ldots, \m_{a+1}, \m_{a+2}, \ldots, \m_s).$$
Therefore, $$\lcm(\m_1 , \ldots , \widehat{\m}_\ell , \ldots , \m_s)
=\lcm(\m_1 , \ldots , \m_s).$$
	
This also contradicts the minimality of the cover $\mathcal{M}$. If $\ell \geq a+1$, then since $\n_i \mid \lcm (\m_1 , \ldots , \m_a)$,
\begin{align}
\nonumber \m_\ell & \mid \lcm (\n_i , \m_{\ell+1} , \ldots , \m_s)\\
\nonumber & \mid \lcm (\m_1 , \ldots , \m_a , \m_{\ell+1} , \ldots , \m_s).
\end{align}
Similarly,
\begin{center}
	$\lcm(\m_1 , \ldots , \m_a , \ldots , \widehat{\m}_\ell , \m_{\ell+1} , \ldots , \m_s)= \lcm(\m_1 , \ldots , \m_s)$,
\end{center}
a contradiction. Thus, $\gcd(\m,\m') \notin I$ and hence $\m$ and $\m'$ are lattice complements in $I$, as required.
	\end{proof}

In fact more is true: the second half of the well
  ordered cover in \cref{p:complements} is itself a well ordered
  cover.

\begin{proposition}
	\label{p:secondpart}
	Let $\mathcal{M}=\{\m_1,\ldots,\m_s\}$ be a well ordered cover
        of a square-free monomial ideal $I$. Define for each $1 \leq a
        \leq s-1$ the monomials
        $$\m=\lcm(\m_1,\ldots,\m_a) \qand
        \m'=\lcm(\m_{a+1},\ldots,\m_s).$$ Then
	
	\begin{enumerate}
	
	\item $\{\m_1,\ldots,\m_a\}$ is a minimal cover of $I_{[\m]}$.
	\item $\{\m_{a+1},\ldots,\m_s\}$ is a well ordered cover of $I_{[\m']}$.
	\item $\be_{s-a,\m'}(S/I)\neq 0$.
	\end{enumerate}
\end{proposition}

\begin{proof}
	Clearly, $\{\m_1,\ldots,\m_a\}$ and $\{\m_{a+1},\ldots,\m_s\}$ are covers of $I_{[\m]}$ and $I_{[\m']}$ respectively. Suppose that $\{\m_1,\ldots,\m_a\}$ is not minimal, i.e. there exists a proper subset 
	\begin{center}
		$\{\m_{i_1},\ldots,\m_{i_k}\} \subset \{\m_1,\ldots,\m_a\}$,
		\end{center}
	that is a cover of $I_{[\m]}$.  In particular,
        there exists $h\in \{1,\ldots,a\} \setminus
        \{i_1,\ldots,i_k\}$ such that $$\m_h \mid \lcm(\m_{i_1} ,
        \cdots , \m_{i_k}).$$ Then
	$$\lcm(\m_1 , \cdots , \widehat{\m}_h , \cdots ,
        \m_s)=\lcm(\m_1 , \cdots , \m_s).$$
        
	This contradicts the minimality of the cover
        $\mathcal{M}$. Thus, $\{\m_1,\ldots,\m_a\}$ is a minimal 
        cover of $I_{[\m]}$. Using a similar argument,
        $\{\m_{a+1},\ldots,\m_s\}$ is also a minimal cover of
        $I_{[\m']}$.
	
	In order to show that $\{\m_{a+1},\ldots,\m_s\}$ is a well
        ordered cover of $I_{[\m']}$, if $\n$ is a minimal generator of
        $I_{[\m']}$ such that $\n \notin
        \{\m_{a+1},\ldots,\m_s\}$, we need to find $a+1 \leq \ell \leq
        s-1$ such that
	\begin{center}
		$\m_\ell \mid \lcm(\n , \m_{\ell+1} , \cdots , \m_s)$.
		\end{center}
	By \cref{p:complements}, $\m$ and $\m'$ are lattice
        complements. As a result, $\n$ cannot be a minimal generator
        of $I_{[\m]}$; in particular, $\n \notin
        \{\m_1,\ldots,\m_a\}$. Since
        $\mathcal{M}=\{\m_1,\ldots,\m_s\}$ is a well ordered cover and
        $\n \notin \{\m_1,\ldots,\m_s\}$, then there exists $1 \leq
        \ell \leq s-1$ such that
	\begin{center}
		$\m_\ell \mid \lcm ( \n , \m_{\ell+1} , \cdots , \m_s)$.
		\end{center}
	Suppose that $\ell \leq a$, so $\ell +1 \leq a+1$. Since $\n$
        is a minimal generator of $I_{[\m']}$, $\n \mid \lcm (
        \m_{a+1} , \cdots , \m_s)$. Therefore,
	\begin{align}
	\nonumber \m_\ell &\mid \lcm ( \n , \m_{\ell+1} , \cdots , \m_{a+1} , \cdots , \m_s)\\
	\nonumber & \mid \lcm ( \m_{\ell+1} , \cdots , \m_s).	\end{align}
	Thus,
	\begin{center}
		$\lcm(\m_1 , \cdots , \widehat{\m}_\ell , \cdots , \m_s)=\lcm(\m_1 , \cdots , \m_s)$,
	\end{center}
contradicting the minimality of the cover $\mathcal{M}$. Hence $a+1
\leq \ell \leq s-1$ which makes $\{\m_{a+1},\ldots,\m_s\}$ a well
ordered cover of $I_{[\m']}$. Using \cref{t:Nursel},
$\be_{s-a,\m'}(S/I_{[\m']})\neq 0$, and hence $\beta_{s-a,\m'}(S/I) \neq
0$.
\end{proof}

\begin{example}
	If $\Delta=\F(I)=\langle xy,yz,xz,za,ab,bc \rangle$, then
        $\mathcal{M}=\{ab,xy,bc,xz\}$ is a well ordered cover.
	
	
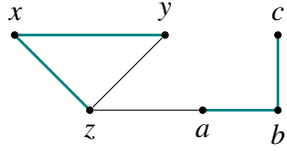
\begin{center}
		\begin{tikzpicture}	
		\coordinate (x) at (-1,1);
		\coordinate (y) at (1,1);
		\coordinate (z) at (0,0);
		\coordinate (a) at (1.5,0);
		\coordinate (b) at (2.5,0);
		\coordinate (c) at (2.5,1);

		\draw [-,line width=1pt,color=teal](x)--(y);
		
		\draw [-,line width=1pt,color=teal](x)--(z);
		
		\draw (x)node[circle,fill,inner sep=1pt,label=above:$x$](x){};
		
		\draw (y)node[circle,fill,inner sep=1pt,label=above:$y$](y){};
		
		\draw (z)node[circle,fill,inner sep=1pt,label=below:$z$](z){};
		
		\draw (y)--(z);
		
		\draw (z)--(a);
		
		\draw [-,line width=1pt,color=teal](a)--(b);
		
		\draw [-,line width=1pt,color=teal](b)--(c);
		
		\draw (a)node[circle,fill,inner sep=1pt,label=below:$a$](a){};
		
		\draw (b)node[circle,fill,inner sep=1pt,label=below:$b$](b){};

		\draw (c)node[circle,fill,inner sep=1pt,label=above:$c$](c){};
		
		\end{tikzpicture}
	\end{center}

	 Then, by \cref{p:secondpart} we have: $\be_{1,xz}(S/I)\neq 0$, $\be_{2,bcxz}(S/I)\neq 0$, $\be_{3,bcxyz}(S/I)\neq 0$ and $\be_{4,abcxyz}(S/I)\neq 0$.
	\end{example}

Under certain extra conditions, the first part of the well ordered
cover is also well ordered, which gives us the subadditivity property
in those cases.  Note that the subadditivity property in the case
$a=1$ below is also covered by a theorem of Herzog and
Srinivasan~\cite{HS}.

\begin{theorem}[{\bf Well ordered covers and subadditivity}]\label{t:main} 
  Let $\mathcal{M}=\{\m_1,\ldots,\m_s\}$ be a well ordered cover
  of a square-free monomial ideal $I$. Define for each $1 \leq a \leq
  s-1$ the monomials
		$$\m=\lcm(\m_1,\ldots,\m_a) \qand
  \m'=\lcm(\m_{a+1},\ldots,\m_s).$$
  Supposed one of the following conditions holds: 
  \begin{enumerate}
  \item $I_{[\m]}=(\m_1,\ldots,\m_a)$ (e.g.  when $a=1$); or
  \item $\gcd(\m,\m')=1$.
	\end{enumerate}
  
  Then $\m$ and $\m^\prime$ are lattice complements in $I=\F(\Delta)$ such that $\be_{a,\m}(S/I)\neq 0$. In particular 
$$t_s \leq t_a+ t_{s-a}.$$
	\end{theorem}

\begin{proof} By \cref{p:secondpart} we already know that  
  $\m_1,\ldots,\m_a$ is a minimal cover for $I_{[\m]}$. We claim that
  if either of the two conditions above hold, then it is a well
  ordered cover.  If $I_{[\m]}=(\m_1,\ldots,\m_a)$, then this is
  trivial, since there is no generator other than
  $\m_1,\ldots,\m_a$. Suppose $\gcd(\m,\m')=1$, and let $\n$ be in the
  minimal monomial generating set of $I_{[\m]}$ and $\n \notin
    \{\m_1,\ldots,\m_a\}$. Then, $\n$ is a minimal generator of $I$ as
    well. Moreover $\m_1,\ldots,\m_s$ is a well ordered cover of $I$,
    and so for some $j \leq s-1$ we have $\m_j \mid
    \lcm(\n,\m_{j+1},\ldots,\m_s)$. On the other hand, since $\n \mid
    \m$ and $\gcd(\m,\m')=1$, we must have $j \leq a-1$ and $$\m_j
    \mid \lcm(\n,\m_{j+1},\ldots,\m_a)$$ which implies that
    $\m_1,\ldots,\m_a$ is a well ordered cover of
    $I_{[\m]}$. Therefore, by \cref{t:Nursel}, $\beta_{a,\m}(S/I_{[\m]})
    \neq 0$, and hence $\beta_{a,\m}(S/I) \neq 0$.
  
    By \cref{p:secondpart} $\m_{a+1},\ldots,\m_s$ is a well ordered cover for $I_{[\m']}$ and $\beta_{s-a,\m'}(S/I_{[\m']}) \neq 0$, hence $\beta_{s-a,\m'}(S/I) \neq 0$.

    Now, since $\m$ and $\m'$ are complements (\cref{p:complements}),
    we have that $$t_s \leq t_a+ t_{s-a},$$ which ends our argument.
\end{proof}

\section{Simplicial Bouquets}\label{s:bouquets}

As proved in~\cite[Proposition~4.3]{EF2}, an example of a well ordered
facet cover for a simple graph is a \emph{strongly disjoint set of
  bouquets} (see~\cite[Definition 1.7]{Z} and
also~\cite[Definitions~2.1~and~2.3]{K}). Bouquets are in general much
easier to identify in graphs than well ordered edge covers, as one
does not need to worry about the order. In this section we define a
simplicial counterpart for a strongly disjoint set of bouquets of
graphs, and show that they form well ordered facet covers and hence
guarantee nonvanishing Betti numbers. Similar to graph bouquets,
simplicial bouquets are easy to spot in a picture, and strongly
disjoint sets of simplicial bouquets often come with more than one
guaranteed order. We then apply the results of the previous section to
examine the subadditivity property in the presence of such simplicial
bouquets. It must be noted that our definition of a simplicial bouquet
is very close to hypergraph bouquets developed in
\cite[Definition~3.1]{KhM}, and almost the same as the hypergraph
bouquets in~\cite[Definition~2.1]{E}.

\begin{definition}[{\bf Simplicial bouquet}]\label{d:bouquet} Let $\Delta$
  be a simplicial complex. A {\bf simplicial bouquet} is a
  subcollection $B=\langle G_1,\ldots, G_t \rangle$ of $\Delta$ such
  that each facet $G_i$ has at least one free vertex in $B$,
  and $$\bigcap_{i=1}^t G_i \neq \emptyset.$$ The nonempty
  intersection of the facets of $B$ is called the ${\bf root}$ of $B$
  and denoted by $\mbox{Root}(B)$.
\end{definition}

A simplicial bouquet in a graph coincides with the usual definition of
bouquets in graphs (see~\cite[Definition 1.7]{Z}).

    The {\bf distance} between two distinct facets $F,F^\prime$ of
    $\Delta$, denoted by $\dist_\Delta(F,F^\prime)$, is the minimum
    length $\alpha$ of sequences of facets of $\Delta$ $$F=F_0,
    F_1,\ldots,F_\alpha=F^\prime \qwhere F_{i-1} \cap F_i \neq
    \emptyset,$$ or $\infty$ if there is no such sequence. We say that
    $F$ and $F^\prime$ are {\bf $3$-disjoint} in $\Delta$ if
    $\dist_\Delta(F,F^\prime) \geq 3$. A subset $\mathcal{E} \subset
    \facets(\Delta)$ is said to be {\bf pairwise $3$-disjoint} if
    every pair of distinct facets $F,F^\prime \in \mathcal{E}$ are
    $3$-disjoint in $\Delta$ (see~\cite[Definitions 2.2 and 6.3]{HV}).

   \begin{definition}[{\bf (Strongly disjoint) set of
            bouquets}] Let $\Delta$ be a simplicial complex. For a set
        $\B=\{B_1,B_2,\dots,B_d\}$ of simplicial bouquets of $\Delta$, define
    $$\facets(\B)=\facets(B_1) \cup \cdots \cup \facets(B_d) \qand
        V(\B)=V(B_1) \cup \cdots \cup V(B_d).$$ Then $\B$ is called
        {\bf strongly disjoint} in $\Delta$ if:
    \begin{enumerate}
      \item $V(B_i) \cap V(B_j)=\emptyset$ for all $i \neq j$, and
      \item we can choose a facet $G_i$ from each $B_i \in \B$
    so that the set $\{G_1,\ldots,G_d\}$ is pairwise $3$-disjoint in
    $\Delta$.
    \end{enumerate}

    We say that $\Delta$ {\bf contains a strongly disjoint
      set of bouquets} if there exists a strongly disjoint set of
    bouquets $\B=\{B_1,\ldots,B_q\}$ of $\Delta$ such that:
    \begin{enumerate}
    \item $V(\Delta)=V(\B)$, and
    \item if $F \in \facets(\Delta)\setminus \facets(\B)$ and $F \cap
      G \neq \emptyset$ for some $G \in \facets(B_i)$ and $i \in
      \{1,\ldots,d\}$, then $\left ( G \smallsetminus \mbox{Root}(B_i)
      \right ) \subseteq F$.
    \end{enumerate}
\end{definition}

\begin{example}\label{e:earlier}
	Let $I=(abc,bcd,cdf,def,eg,fg,gh,hi,gi,fi,gx,gy)$ and
        $\Delta=\F(I)$. It is easy to check that $\Delta$ contains a
        strongly disjoint set of bouquets $\B=\{B_1,B_2\}$ where the
        bouquets $$B_1=\langle bcd, abc \rangle \qand B_2=\langle
        gy,gx,ge,gf,gh,gi \rangle$$ are highlighted below, and the set
        of facets $\{abc,gx\}$ is pairwise $3$-disjoint.
	
	
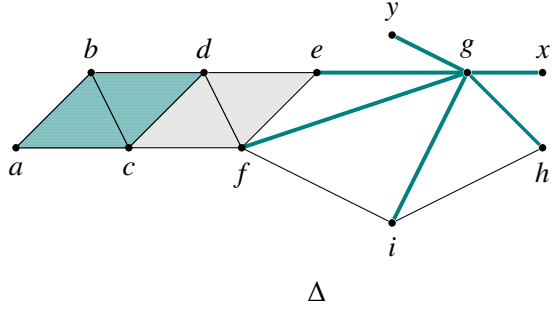
\begin{center}
	\begin{tikzpicture}
	
	\coordinate (a) at (-3,0);
	\coordinate (b) at (-2,1);
	\coordinate (c) at (-1.5,0);
	\coordinate (d) at (-0.5,1);
	\coordinate (e) at (1,1);
	\coordinate (w) at (1,-1.5);
	\coordinate (f) at (0,0);
	\coordinate (g) at (3,1);
	\coordinate (i) at (2,-1);
	\coordinate (y) at (2,1.5);
	
	\coordinate (h) at (4,0);
	
	\coordinate (x) at (4,1);
	
	\draw (a)--(b)(a)--(c);

	\draw (g) node[circle,fill,inner sep=1pt,label=above:$g$](g){};
	\draw (h) node[circle,fill,inner sep=1pt,label=below:$h$](h){};
	
	\draw [-,line width=1.3pt,color=teal](g)--(h);
	
	\draw (i) node[circle,fill,inner sep=1pt,label=below:$i$](i){};
	\draw [-,line width=1.5pt,color=teal](g)--(i);
	\draw [-,line width=1.5pt,color=teal](g)--(f);
	\draw [-,line width=1.5pt,color=teal](g)--(e);
	
	\draw (x) node[circle,fill,inner sep=1pt,label=above:$x$](x){};
	
	\draw (y) node[circle,fill,inner sep=1pt,label=above:$y$](y){};
	
	\draw [-,line width=1.5pt,color=teal](g)--(x);
	
	\draw [-,line width=1.5pt,color=teal](g)--(y);

	\draw (i)--(h)(i)--(f);
	
	\draw [top color=teal!50,bottom color=teal!50](a)node[circle,fill,inner sep=1pt,label=below:$a$](a){}--(b)node[circle,fill,inner sep=1pt,label=above:$b$](b){}--(c)--(a);
	
	\draw [top color=teal!50,bottom color=teal!50](b)--(c)--(d)--(b);
	
	\draw [top color=gray!20,bottom color=gray!20](d)--(c)node[circle,fill,inner sep=1pt,label=below:$c$](c){}--(f)--(d)node[circle,fill,inner sep=1pt,label=above:$d$](d){};
	
	\draw [top color=gray!20,bottom color=gray!20](d)--(f)node[circle,fill,inner sep=1pt,label=below:$f$](f){}--(e)node[circle,fill,inner sep=1pt,label=above:$e$](e){}--(d);

	\draw (w)node[label=below:$\Delta$](w){};
	\end{tikzpicture}  
	\end{center}

\end{example}

    The following statement is a generalization
    of~\cite[Proposition~4.3]{EF2}.
    
\begin{theorem}[{\bf Strongly disjoint set of bouquets and well ordered
facet covers}]
	\label{t:bouquetsstrongly}
	Let $\Delta$ be a simplicial complex which contains a
        strongly disjoint set of bouquets $\B=\{B_1,\ldots,B_d\}$. For
        each $q \in \{1,\ldots,d\}$ suppose
          $$\facets(B_q)=\{G_1^q,G_2^q,\ldots,G_{b_q}^q,G_q\}$$ where
        $\{G_1,\ldots,G_d\}$ is pairwise $3$-disjoint in
        $\Delta$. Then for any permutation $k_1,k_2,\ldots,k_d$ of the
        integers $1,\ldots,d$, the sequence of facets
			$$\underbrace{G_1^{k_1} \; , \; \ldots \; , \;
          G_{b_{k_1}}^{k_1} \; , \; G_{k_1}}_{\facets(B_{k_1})} \; , \;
        \underbrace{G_1^{k_2} \; , \; \ldots \; , \; G_{b_{k_2}}^{k_2} \; , \;
          G_{k_2}}_{\facets(B_{k_2})} \; , \; \ldots \; , \; \underbrace{G_1^{k_d} \; , \;
          \ldots \; , \; G_{b_{k_d}}^{k_d} \; , \; G_{k_d}}_{\facets(B_{k_d})}$$ form a
        well ordered facet cover of $\Delta$. 
	\end{theorem}

\begin{proof} Observe that we can identify each monomial generator
  $x_1\ldots x_a$ of $\F(\Delta)$ with the facet $\{x_1,\ldots, x_a\}$
  of $\Delta$. The condition that $V(\Delta)=V(\B)$ and that every
  facet of $\B$ has a free vertex guarantees that $\facets(\B)$ is a
  minimal facet cover of $\Delta$.  Suppose $F \in \facets(\Delta)
  \setminus \facets(\B)$. Since the set $\{G_{k_1},\ldots,G_{k_d}\}$
  is pairwise $3$-disjoint in $\Delta$, $F$ can intersect at most one
  of $\{G_{k_1},\ldots,G_{k_d}\}$, and so at least one vertex of $F$
  does not belong to $\bigcup_{i=1}^d G_{k_i}$. Therefore for some $q
  \in \{1,\ldots,d \}$ and $j \in \{1, \ldots,b_{k_q}\}$ we have
  $G_j^{k_q} \cap F \neq \emptyset$ and so $$F \supseteq G_j^{k_q} \smallsetminus
  \mbox{Root}(B_{k_q}).$$   Hence
  $$G_j^{k_q} = (G_j^{k_q} \smallsetminus
  \mbox{Root}(B_{k_q})) \cup
  \mbox{Root}(B_{k_q}) \subseteq F \cup G_{k_q} \subseteq
  F \cup G_{j+1}^{k_q} \cup \cdots \cup G_{b_{k_q}}^{k_q} \cup G_{k_q}
  \cup G_1^{k_{q+1}} \cup \cdots \cup G_{k_d}$$ which implies that the
  sequence above is a well ordered facet cover of
  $\Delta$.
	\end{proof}

Equivalently, strongly disjoint simplicial bouquets
  produce well ordered covers for facet ideals.

  \begin{example} For the ideal in \cref{e:earlier},
    \cref{t:bouquetsstrongly} states that both
 $$bcd, abc, gy,ge,gf,gh,gi,gx \qand gy,ge,gf,gh,gi,gx, bcd, abc$$ are
    well ordered covers. Note that there are many other
      options for well ordered covers, obtained by reordering all
      facets (except for the last one) of each bouquet, or picking
      another set of pairwise $3$-disjoint facets from the set of
      bouquets.
  \end{example}

The following is a generalization of~\cite[Theorem
    3.1]{K}. From here onwards we assume $S$ is a polynomial ring over
a field generated by variables which are vertices of the simplicial
complex $\Delta$.

\begin{corollary}[{\bf Betti numbers from simplicial bouquets}] 
	\label{c:nonzeroBetti}
	Let $\Delta$ be a simplicial complex, $W \subseteq V(\Delta)$
        and suppose the induced subcollection $\Delta_{[W]}$ contains
        strongly disjoint set of bouquets $\B$ with $|\facets(\B)|=i$
        and $|W|=|V(\B)|=j$. Then
        $\beta_{i,j}(S/\F(\Delta))\neq 0$.
\end{corollary}

\begin{proof} As discussed in \cref{s:subadditivity},
  $\beta_{i,j}(S/\F(\Delta))\neq 0$ if and only if
  $\beta_{i,j}(S/\F(\Delta_{[W]}))\neq 0$ for some $W \subseteq
  V(\Delta)$ with $|W|=j$. The statement is now a direct consequence
  of \cref{t:Nursel,t:bouquetsstrongly}.
\end{proof}

\begin{theorem}[{\bf Subadditivity from simplicial bouquets}]
	\label{t:bouquets}
	Let $\Delta$ be a simplicial complex and
        $I=\F(\Delta)$. Suppose $\Delta$ contains a strongly disjoint
        set of bouquets $\B$.  Let $\B=\B' \sqcup \B''$ be a partition
        of $\B$ into disjoint subsets $\B'$ and $\B''$, and
        let $$b'=|\facets(\B')|, \quad b''=|\facets(\B'')|, \qand
        b=b'+b''=|\facets(\B)|.$$ Then the monomials $$\m =\prod_{x
          \in V(\B')} x \qand \m'=\prod_{x \in V(\B'')}x $$ are
        lattice complements in $LCM(I)$ and
        $$\beta_{b',\m}(S/I)\neq 0 \qand \beta_{b'',\m'}(S/I)\neq 0$$
        In particular, $$t_{b} \leq t_{b'} + t_{b''}.$$
\end{theorem}

  \begin{proof} As any bouquet of $\B$ is either in $\B'$ or
  in $\B''$, and no two distinct bouquets share any vertices,
  $\gcd(\m,\m')=1$.  The claim now follows directly from \cref{t:main}
  and \cref{t:bouquetsstrongly}.
  \end{proof}

\begin{example}
	\label{e:runexample}
	Let $I=(ax,ay,bz,bv,bw,cu,cg,yz,az)$ and $G=\F(I)$. It is easy
        to check that $\B=\{\langle ax,ay\rangle, \langle bz,bv,bw
        \rangle,\langle cu,cg \rangle \}$ (highlighted below) is a
        strongly disjoint set of bouquets in $G$ where $\{ax,bv,cu\}
        \subset \facets(\B)$ is a set of pairwise $3$-disjoint in $G$.
	
	\begin{center}
		\begin{tikzpicture}

			\coordinate (a) at (0,0);
			\coordinate (x) at (-1,1);
			\coordinate (y) at (0.5,1);
			\coordinate (b) at (3,0);
			\coordinate (z) at (2,1);
			\coordinate (v) at (3,1);
			\coordinate (w) at (4,1);
			\coordinate (c) at (6,0);
			\coordinate (u) at (5,1);
			\coordinate (g) at (6,1);
			
			
			
			\draw (a)--(x)(a)--(y)(b)--(z)(b)--(v)(b)--(w)(c)--(u)(c)--(g)(y)--(z)(a)--(z);

			\draw (a) node[circle,fill,inner sep=1pt,label=below:$a$](r){};
			\draw (x) node[circle,fill,inner sep=1pt,label=above:$x$](x){};
			\draw (y) node[circle,fill,inner sep=1pt,label=above:$y$](y){};
			\draw (z) node[circle,fill,inner sep=1pt,label=above:$z$](z){};
			
			\draw (v) node[circle,fill,inner sep=1pt,label=above:$v$](v){};
			
			\draw (w) node[circle,fill,inner sep=1pt,label=above:$w$](w){};
			
			\draw (u) node[circle,fill,inner sep=1pt,label=above:$u$](u){};
			
			\draw (g) node[circle,fill,inner sep=1pt,label=above:$g$](g){};
			
			\draw (b) node[circle,fill,inner sep=1pt,label=below:$b$](b){};
			
			\draw (c) node[circle,fill,inner sep=1pt,label=below:$c$](c){};

			
			\draw [-,line width=1.5pt,color=teal](a)--(x);
			\draw [-,line width=1.5pt,color=teal](a)--(y);
			\draw [-,line width=1.5pt,color=teal](b)--(z);
			\draw [-,line width=1.5pt,color=teal](b)--(v);
			\draw [-,line width=1.5pt,color=teal](b)--(w);
			\draw [-,line width=1.5pt,color=teal](c)--(u);
			\draw [-,line width=1.5pt,color=teal](c)--(g);
			
			
			

			
			
			
			
			
		\end{tikzpicture}  
	\end{center}
        According to \cref{c:nonzeroBetti} (or (\cite[Theorem 3.1]{K}
          since this is the case of bouquets in graphs)
        $\beta_{7,10}(S/I)\neq 0$, so that $t_7=10$. Using the
        following order of $\B$ $$\underbrace{ay\; , \;
          ax}_{\facets(\langle ax,ay \rangle)}\; , \; \underbrace{bz\;
          , \; bw\; , \;bv}_{\facets(\langle bz,bv,bw \rangle)}\; , \;
        \underbrace{cg\; , \;cu}_{\facets(\langle cu,cg \rangle)}$$
        and setting $\m=axy=\lcm(ay,ax)$ and
        $\m^\prime=bzvwcug=\lcm(bz,bw,bv,cg,cu)$, by \cref{t:bouquets} we have
        $\beta_{2,\m}(S/I) \neq 0$ and $\beta_{5,\m^\prime}(S/I) \neq
        0$. On the other hand, using the
        following reordering of $\B$
        $$\underbrace{ay\; , \; ax}_{\facets(\langle ax,ay \rangle)}\;
        , \;\underbrace{cg\; , \;cu}_{\facets(\langle cu,cg
          \rangle)}\; , \;\underbrace{bz\; , \; bw\; , \;
          bv}_{\facets(\langle bz,bv,bw \rangle)}$$ and setting
        $\m=acxyug=\lcm(ay,ax,cg,cu)$ and
        $\m^\prime=bzvw=\lcm(bz,bw,bv)$, by \cref{t:bouquets}
          we also get $\beta_{4,\m}(S/I) \neq 0$ and
        $\beta_{3,\m^\prime}(S/I) \neq 0$. As a result,
        $$t_7<t_2+t_5=12 \qand t_7<t_3+t_4=13$$
        which can be confirmed by using the Betti table for $S/I$.

$$\begin{array}{rllllllllll}
	&0 &1 &2 &3 &4 &5 &6 &7 \\
	\mbox{total}:&1 &9 &28 &44 &40 &22 &7  &1\\
	0:&1 &. &. &. &. &. &. &. \\
	1:&. &9 &10 &3 &. &. &. &.  \\
	2:&. &. &18 &33 &20 &4 &. &.  \\
	3:&. &. &. &8 &20 &18 &7  &1
\end{array}$$

	\end{example}

\begin{example} 
	\label{e:runexample-bouquets}For $I$ and  $\B$ in \cref{e:earlier},  since
  $|\facets(B_1)|=2$ and $|\facets(B_2)|=6$,  and the two monomials
  $abcd$ and $efghixy$ are complements in $\LCM(I)$, it follows that $\beta_{6,7}(S/I)
  \neq 0$ and $\beta_{2,4}(S/I) \neq 0$  so that $t_{8} \leq t_2 +
  t_6$.
\end{example}

\section{Reordering well ordered covers}\label{s:reordering}

As we saw in the \cref{s:subadditivity}, the tail end of every well
ordered cover is itself a well ordered cover, and therefore guarantees
a nonvanishing Betti number in a certain homological degree. For
example, let
\begin{equation}\label{e:cover}  \m_1,\ldots,\m_s 
\end{equation}
be a well ordered cover of $I$.  As a result of \cref{p:secondpart} we
have that
\begin{equation}\label{e:first}
  \beta_{i,\m'}(S/I) \neq 0 \quad \mbox{for} \quad 1 \leq i \leq s
  \mbox{ and } \m'=\lcm(\m_{s-i+1},\ldots, \m_s).
\end{equation}

On the other hand, we also know by \cref{p:complements} that for a
fixed $i$, if $\m=\lcm(\m_1,\ldots, \m_{s-i})$ and $\m'$ is as in
\eqref{e:first}, then $\m$ and $\m'$ are complements. So, if we
additionally have
\begin{equation}\label{e:second}
  \beta_{s-i,\m}(S/I)\neq 0
\end{equation}
then \eqref{e:first} and \eqref{e:second} together would imply
that
\begin{equation}\label{e:third}
  t_s \leq t_i+t_{s-i}.
\end{equation}

In this section we consider extra conditions under which
\eqref{e:second} would hold. One particular situation is when the
monomials $\m_1,\ldots, \m_{s-i}$ in the well ordered cover in
\eqref{e:cover} could be shifted to the end of the ordering, so
that $$\m_{s-i+1},\ldots, \m_s, \m_1,\ldots, \m_{s-i}$$ is also a well
ordered cover of $I$. In this case, by \cref{p:secondpart} we would have
\eqref{e:second} automatically, and therefore more cases of the subadditivity
inequality \eqref{e:third}  could easily follow.

\begin{example}\label{e:modification}
Consider the ideal $I$ and $\Delta=\F(I)$ as in
  \cref{e:earlier} and consider the well ordered cover

\begin{center}
	\def\mytext{$\mathcal{M}:$ \hspace{5mm}\stackanchor{$\m_1$}{$gy$}\hspace{3mm} $,$ \hspace{3mm}\stackanchor{$\m_2$}{$gx$} \hspace{3mm}$,$\hspace{3mm}\stackanchor{$\m_3$}{$ge$}\hspace{3mm}$,$ \hspace{3mm}\stackanchor{$\m_4$}{$gf$}\hspace{3mm}$,$ \hspace{3mm}\stackanchor{$\m_5$}{$bcd$}\hspace{3mm}$,$ \hspace{3mm}\stackanchor{$\m_6$}{$gh$}\hspace{3mm}$,$ \hspace{3mm}\stackanchor{$\m_7$}{$gi$}\hspace{3mm}$,$\hspace{3mm}\stackanchor{$\m_8$}{$abc$}}
	\mytext
\end{center}

Name $\n_1=cdf$, $\n_2=def$, $\n_3=fi$ and
$\n_4=hi$ and for each $k=1,2,3,4$, set $$\alpha_k=\max \{ j \st \m_j
\mid \lcm (\n_k , \m_{j+1} , \cdots , \m_8)\}.$$ Then
$\alpha_1=\alpha_2=5$, $\alpha_3=4$ and $\alpha_4=6$. Let
$\ell=\min(\alpha_1,\alpha_2,\alpha_3,\alpha_4)=4$. Using $\ell$, we
can modify the order in $\mathcal{M}$ as follows:

 \begin{center}
 	\def\mytext{$\mathcal{M}':$ \hspace{5mm}\stackanchor{$\m_4$}{$gf$}\hspace{3mm}$,$ \hspace{3mm}\stackanchor{$\m_5$}{$bcd$} \hspace{3mm}$,$\hspace{3mm}\stackanchor{$\m_6$}{$gh$}\hspace{3mm}$,$ \hspace{3mm}\stackanchor{$\m_7$}{$gi$}\hspace{3mm}$,$ \hspace{3mm}\stackanchor{$\m_8$}{$abc$}\hspace{3mm}$,$ \hspace{3mm}\stackanchor{$\m_1$}{$gy$}\hspace{3mm}$,$ \hspace{3mm}\stackanchor{$\m_2$}{$gx$}\hspace{3mm}$,$ \hspace{3mm}\stackanchor{$\m_3$}{$ge$}}
 	
 	\mytext
 	
 \end{center} 

It is easy to see that $\mathcal{M}'$ is also a well ordered cover of $I$.
 \end{example}

The example above is a special case of a simple observation, stated in
\cref{p:modification}, which is often powerful enough to prove the
subadditivity inequality \eqref{e:third} for all $i$.

\begin{proposition}[{\bf Reordering well ordered covers}]
	\label{p:modification}
	Let $I=(\m_1,\ldots,\m_s,\n_1,\ldots,\n_k)$ be a square-free
        monomial ideal which has $\m_1,\ldots,\m_s$ as
        a well ordered cover. Define for each $1 \leq i \leq k$,
	\begin{equation}\label{e:alpha}
          \alpha_i=\max \big \{j \st \m_j \mid \lcm(\n_i , \m_{j+1} ,
        \cdots , \m_s) \big \}.
        \end{equation}

        If $\ell=\min(\alpha_1,\ldots,\alpha_k)>1$, then
  \begin{enumerate}
  \item  for every $i\in \{2,\ldots,\ell\}$ the sequence
		$$\m_i,\m_{i+1},\ldots,\m_s,
    \m_1,\ldots,\m_{i-1}$$ is a well ordered cover of $I$; 
  \item $t_s \leq t_{s-i}+t_i \quad$ for $1 \leq i \leq \ell -1$.
            \end{enumerate}
      
	\end{proposition}

\begin{proof} Statement 1 follows directly from the definition of
    well ordered covers.  For the second statement, by part 1, for $1 \leq i <\ell$, we have
    $$\m_{i+1},\ldots,\m_s,\m_1,\ldots,\m_{i}$$
    is a well ordered cover for $I$. By \eqref{e:alpha}
    $$\m_{i+1},\ldots,\m_s$$
    is a well ordered cover of $I_{[\m]}$ where
    $\m=\lcm(\m_{i+1},\ldots,\m_s)$.
    By \cref{p:secondpart} 
    $$\m_1,\ldots,\m_i$$ is a well ordered cover of
    $I_{[\m']}$ where
    $$\m'=\lcm(\m_1,\ldots,\m_i).$$
    By \cref{p:complements}, $\m$ and $\m'$ are complements in $\LCM(I)$,
    and by \cref{t:Nursel} $$\beta_{s-i,\m} \neq 0 \qand \beta_{i,\m'}
    \neq 0,$$ which together imply that $t_s \leq t_{s-i}+t_{i}$.
\end{proof}

\begin{example}
  The ideal $I$ from \cref{e:earlier} has the following
    well ordered cover
    $$\M : gy\; ,\; gx \;,\; ge \;,\; gf \;,\; bcd \;,\; gh \;,\; gi
     \;,\; abc.$$ In particular $\beta_{8,11}(S/I)\neq 0$ and
     $t_8=11$. The reordering in \cref{p:modification} of
     $\mathcal{M}$ yields the following well ordered cover

$$\M' : gf \; , \; bcd \; , \; gh \; , \; gi \; , \; abc \; , \; gy \;
     , \; gx \; , \; ge,$$ where $\ell=4$. By part 2 of
     \cref{p:modification} and using the new well  ordered cover
     $\mathcal{M}^\prime$, we have:
\begin{enumerate}
\item $t_8\leq t_7+t_1$. Here we take
  $\m=abcdfghixy=\lcm(gf,bcd,gh,gi,abc,gy,gx)$ and $\m'=ge$.
\item $t_8 \leq t_6+t_2$. Here we take
  $\m=abcdfghiy=\lcm(gf,bcd,gh,gi,abc,gy)$ and
  $\m'=gex=\lcm(gx,ge)$. Note that this inequality was also done in
  \cref{e:runexample-bouquets} using simplicial bouquets.
\item $t_8 \leq t_5+t_3$. Here we take
  $\m=abcdfghi=\lcm(gf,bcd,gh,gi,abc)$ and $\m'=gexy=\lcm(gy,gx,ge)$.
\end{enumerate}

The only remaining case is $a=b=4$. If we take
$\m=bcdfghi=\lcm(gf,bcd,gh,gi)$ and
$\m'=abcgexy=\lcm(abc,gy,gx,ge)$, then $\{abc,gy,gx,ge\}$ is a well ordered cover of $I_{[\m']}$ by \cref{p:secondpart}, and it is
easy to check that $\{gf,bcd,gh,gi\}$ is a well ordered cover of
$I_{[\m]}$. Hence $\be_{4,\m}(S/I)\neq 0$ and $\be_{4,\m'}(S/I)\neq 0$
by \cref{t:Nursel}. Also as $\m$ and $\m'$ are lattice complements (by
\cref{p:complements}), we get $t_4+t_4 \geq \deg(\m)+\deg(\m')> 11=t_8$.

\end{example}


\end{document}